\documentclass[12pt,leqno]{article}
\date{}
\usepackage{amsfonts}
\usepackage{bbm}
\usepackage{amsmath}
\usepackage{mathrsfs}
\usepackage{amssymb}
\usepackage{amsmath,color}
\usepackage{amsthm}
\usepackage{amstext}
\usepackage{amsopn}
\usepackage{color}
\usepackage{texdraw}
\usepackage{graphicx}
\usepackage{multirow}
\usepackage{lscape}
\usepackage{float}
\usepackage{txfonts}

\usepackage{subfigure}
\newtheorem{theorem}{Theorem}[section]
\newtheorem{lemma}[theorem]{Lemma}

\newtheorem{proposition}[theorem]{Proposition}
\newtheorem{definition}[theorem]{Definition}
\theoremstyle{plain}
\newtheorem{remark}[theorem]{Remark}
\numberwithin{equation}{section}

\newcommand{\iE}{\int_{\BB}}
\newcommand{\da}{\frac{dx_1}{x_1}dx'}

\newcommand{\p}{\partial}

\newcommand{\dt}{\tilde}

\newcommand{\la}{\lambda}
\newcommand{\nE}{\nabla_{\BB}}

\newcommand{\na}{\nabla}

\newcommand{\R}{\mathbb{R}}

\newcommand{\fr}{\frac}

\newcommand{\BB}{\mathbb{B}}

\newcommand{\ml}{\mathcal{L}}

\begin{document}

\title{Higher regularity and finite time blow-up to nonlocal pseudo-parabolic equation with conical degeneration\footnote{This work is supported by National Natural Science Foundation of China (12201567), and Scientific Research Fund of Zhejiang Provincial Education Department (Y202249802), and the Young Doctor Program of Zhejiang Normal University (2021ZS0802)}}

\author{Jingbo Meng\footnote{e-mail: meng4897@zjnu.edu.cn},\quad Guangyu Xu\footnote{Corresponding author, e-mail: guangyuswu@126.com,\ xuguangyu@zjnu.edu.cn}\\
{\small  School of Mathematical Sciences, Zhejiang Normal University \hfill{\ }}\\
\ \ {\small Jinhua, 321004, P.R. China}}

\date{}
\maketitle

\begin{abstract}
  This paper deals with the initial-boundary value problem to a nonlocal semilinear pseudo-parabolic equation with conical degeneration, which has been studied in [Global well-posedness for a nonlocal semilinear pseudo-parabolic equation with conical degeneration, J. Differential Equations, 2020, 269(5): 4566--4597]. We first improve the regularity of the weak solution, and then study finite time blow-up phenomenon for the problem. Our initial condition for blow-up only depends on Nehari functional and the conservative integral, which suggests that the assumption of initial energy functional in the original paper can be removed.
\end{abstract}
\maketitle

\section{Introduction}

Let $X$ be a $(n-1)$-dimensional closed compact $C^\infty$-smooth manifold, which is regarded as the local model near the conical points. Take $\BB=[0,1)\times X$, then $\p\BB=\{0\}\times X$, near $\p\BB$ we use the coordinates $(x_1, x')=(x_1, x_2, \cdots, x_n)$ for $0\leq x_1<1, x'\in X$. We denote by $\BB_0$ the interior of $\BB$.

In this paper, we consider the following initial-boundary value problem for a nonlocal semilinear pseudo-parabolic equation with conical degeneration
\begin{equation}\label{physic}
    \left\{
    \begin{aligned}
     & u_t-\Delta_\mathbb{B}u_t-\Delta_\mathbb{B}u=|u|^{p-1}u-\frac{1}{|\mathbb{B}|}\iE |u|^{p-1}u\frac{dx_1}{x_1}dx',\;\;\;&& t>0,\ x\in\BB_0,\\
           &\nabla_{\BB}u\cdot\nu=0,\;\;\;\;&& t\geq0,\ x\in\p\BB,\\
            &u(x,0)=u_{0}(x),\;\;\;&&x\in\BB_0,
    \end{aligned}
    \right.
\end{equation}
where $u_0$ is a nontrivial function and belongs to the weighted Sobolev space $\tilde{\mathcal{H}}^{1,\fr{n}{2}}_{2,0}(\BB)$. Constant $p$ satisfies
\begin{equation*}
2<p+1<+\infty,\ \mbox{if}\ n=2, \quad 2<p+1<\fr{2n}{n-2}=:2^*,\ \mbox{if}\ n\geq3,
\end{equation*}
here $2^*$ is the critical cone Sobolev exponent. The Fuchsian type Laplace operator is defined as
 \begin{equation*}
  \Delta_\mathbb{B}=(x_1\p_{x_1})^2+\p^2_{x_2}+\cdots+\p^2_{x_n},
\end{equation*}
which is an elliptic operator with conical degeneration on the boundary $x_1=0$, and the corresponding gradient operator is $\nabla_{\BB}=(x_1\p_{x_1}, \p_{x_2}, \cdots, \p_{x_n}), \nu$ is the unit normal vector pointing toward the exterior of $\BB$. The detail research of manifold with conical singularities and the corresponding cone Sobolev spaces can be found in \cite{Chen2011Existence,Chen2012Cone,Coriasco2007Realizations,Schrohe2001Ellipticity,Schulze1999Boundary}.

{\color{blue}{To date, research on problems for pseudo-parabolic equations is rapidly developing in connection with the needs of modeling and controlling of processes in thermophysics, hydrodynamics and continuum mechanics. Pseudo-parabolic equations are characterized by the occurrence of a time derivative appearing in the highest order term and describe a variety of important physical processes, such as the seepage of homogeneous fluids through a fissured rock \cite{b}, the heat conduction involving two temperatures \cite{cc}. When the unknown $u(x, t)$ is typically the amplitude or velocity, the corresponding pseudo-system denotes the unidirectional propagation of nonlinear, dispersive, long waves \cite{t}, while $u(x, t)$ represents the population density, the corresponding pseudo-problem can describe the aggregation of populations \cite{p}.}} The authors in \cite{Showalter1970Pseudoparabolic,tt} investigated the initial-boundary value problem and Cauchy problem for the linear pseudo-parabolic equation and established the existence and uniqueness of solutions. After those precursory results, there are many paper (for example \cite{bbb,cao,jism,kkk,kk,pp}) studied nonlinear pseudo-parabolic equations, like semilinear pseudo-parabolic equations and even including singular and degenerate pseudo-parabolic equations.

Problem \eqref{physic} has been proposed and studied in \cite{dihuafei}, the authors considered the weak solution of the problem. For convenience, we give some abbreviations,
\begin{equation*}\begin{split}
\|u\|_{\ml^{\fr{n}{p}}_p{(\BB)}}=\left(\iE |u|^p\frac{dx_1}{x_1}dx'\right)^{\fr{1}{p}},\ \ \|u\|_{\tilde{\mathcal{H}}^{1,\fr{n}{2}}_{2,0}(\BB)}^2=\|u\|_{\ml^{\fr{n}{2}}_2{(\BB)}}^2+\|\nabla_{\BB}u\|_{\ml^{\fr{n}{2}}_2{(\BB)}}^2,
\end{split}\end{equation*}
and
\begin{equation*}
(u, v)=\iE uv\frac{dx_1}{x_1}dx'.
\end{equation*}

The following definition for the weak solution to problem \eqref{physic} can be found in \cite[Definition 2.3]{dihuafei}.

\begin{definition}\label{def}
Function $u\in L^\infty\left(0, T; \tilde{\mathcal{H}}^{1,\fr{n}{2}}_{2,0}(\BB)\right)$ with $u_t\in L^2\left(0, T; \tilde{\mathcal{H}}^{1,\fr{n}{2}}_{2,0}(\BB)\right)$ is called a weak solution of problem \eqref{physic}, if $u(x, 0)=u_0\in\tilde{\mathcal{H}}^{1,\fr{n}{2}}_{2,0}(\BB)\backslash\{0\}$ and $u$ satisfies \eqref{physic} in the following distribution sense,
\begin{equation}\label{rj}
\left(u_t, \varphi\right)+\left(\nabla_{\BB}u_t, \nabla_{\BB}\varphi\right)+\left(\nabla_{\BB}u, \nabla_{\BB}\varphi\right)=
\left(|u|^{p-1}u-\frac{1}{|\mathbb{B}|}\iE |u|^{p-1}u\frac{dx_1}{x_1}dx',\ \varphi\right),
\end{equation}
for any $\varphi\in\tilde{\mathcal{H}}^{1,\fr{n}{2}}_{2,0}(\BB)$ and a.e. $t\in (0,T)$.
\end{definition}

Roughly speaking, the authors in \cite{dihuafei} combined the modified methods of Galerkin approximation, potential well \cite{Liu2003On,Payne1975Saddle,Sattinger1968On}, concavity \cite{Levine1973Some,Tsutsumi1972On} and the classical variational theory to prove the existence of global solutions and finite time blow-up for problem \eqref{physic} at the three different initial energy levels, i.e., subcritical initial energy $J(u_0)<d$, critical initial energy $J(u_0)=d$ and high initial energy $J(u_0)>d$, here $J(u)$ is the potential energy associated with problem \eqref{physic},
\begin{equation*}\begin{split}
J(u)=\fr{1}{2}\iE |\nE u|^2\da-\fr{1}{p+1}\iE |u|^{p+1}\da,
\end{split}\end{equation*}
and $d$ is the depth of potential well, which can be defined by
\begin{equation*}\begin{split}
d=\inf\left\{\sup_{\la\geq0}J(\la u),\ u\in\tilde{\mathcal{H}}^{1,\fr{n}{2}}_{2,0}(\BB), \iE|\nE u|^2\da\not=0\right\}.
\end{split}\end{equation*}
The energy functional $J(u)$ plays very important role in the whole analysis process, an other key functional is the Nehari functional
\begin{equation*}\begin{split}
I(u)=\iE |\nE u|^2\da-\iE |u|^{p+1}\da.
\end{split}\end{equation*}
The importance of $J(u)$ and $I(u)$ comes from the fact that they have considerable fine characters. In fact, we can infer from \cite[(2.8)]{dihuafei}, i.e.,
\begin{equation}\label{cw2}
\frac{d}{dt}J(u)=-\|u_t\|^2_{\tilde{\mathcal{H}}^{1,\fr{n}{2}}_{2,0}(\BB)},
\end{equation}
that the energy functional is nonincreasing with respect to $t$. Moreover, by \cite[(4.17)]{dihuafei}, i.e.,
\begin{equation}\begin{split}\label{cw3}
 \frac{1}{2}\frac{d}{dt}\|u\|^2_{\tilde{\mathcal{H}}^{1,\fr{n}{2}}_{2,0}(\BB)}&=\iE\left( -|\nE u|^2+|u|^{p+1}\right)\da-\frac{S(u_0)}{|\BB|}\|u\|^p_{\mathcal{L}^{\fr{n}{p}}_{p}(\BB)}\\
&=-I(u)-\frac{S(u_0)}{|\BB|}\|u\|^p_{\mathcal{L}^{\fr{n}{p}}_{p}(\BB)},
\end{split}\end{equation}
we know there is a well relation between $\frac{d}{dt}\|u\|^2_{\tilde{\mathcal{H}}^{1,\fr{n}{2}}_{2,0}(\BB)}$ and $-I(u)$, and this character is a powerful tool in the proofs of main results in \cite{dihuafei}. Moreover, in \cite[(2.6)]{dihuafei}, the author claimed
\begin{equation}\begin{split}\label{cw1}
\frac{d}{dt}\iE u\frac{dx_1}{x_1}dx'=\iE u_t\frac{dx_1}{x_1}dx'=0,
\end{split}
\end{equation}
this is also a significant property to the weak solution of problem \eqref{physic} since which implies that $\iE u(t)\frac{dx_1}{x_1}dx'$ remains constant in time for all $t\in (0, T)$, namely,
\begin{equation}\label{shd}
S(u_0):=\iE u\frac{dx_1}{x_1}dx'=\iE u_0\frac{dx_1}{x_1}dx'.
\end{equation}

However, in order to get \eqref{cw2}, the authors in \cite{dihuafei} replaced $\varphi$ by $u_t$ in \eqref{rj}, similarly, they took $\varphi$ by $u$ in \eqref{rj} to obtain \eqref{cw3}. It must be point out that such two treatments only hold formally due to the lack of regularity of weak solution. By the definition of weak solution, for any $t\in (0, T)$, there holds
\begin{equation}\label{qjxd}
u\in L^\infty\left(0, T; \tilde{\mathcal{H}}^{1,\fr{n}{2}}_{2,0}(\BB)\right)\ \ \mbox{and}\ \ u_t\in L^2\left(0, T; \tilde{\mathcal{H}}^{1,\fr{n}{2}}_{2,0}(\BB)\right).
\end{equation}
{\color{blue}{Then we can infer from \cite{Cazenave,Evans}, or as proved in \cite[Page 17]{dihuafei} by Aubin-Lions-Simon lemma (see \cite[Corollary 4]{Simon}) to obtain }}
\begin{equation}\label{zjsk}
u\in C\left([0, T); \tilde{\mathcal{H}}^{1,\fr{n}{2}}_{2,0}(\BB)\right).
\end{equation}
Further, by the continuous embedding $\tilde{\mathcal{H}}^{1,\fr{n}{2}}_{2,0}(\BB)\hookrightarrow C(\BB)$ we have
\begin{equation}\label{zjsk2}
u\in C\left([0, T); C(\BB)\right).
\end{equation}
To the best of our knowledge, the highest regularity of weak solution with respect to $t$ given by \cite{dihuafei} are \eqref{zjsk} and \eqref{zjsk2}. Nevertheless, neither of them can not lead to $\|u\|^2_{\tilde{\mathcal{H}}^{1,\fr{n}{2}}_{2,0}(\BB)},\ J(u),\ \iE u\frac{dx_1}{x_1}dx' \in C^1(0, T)$.

Actually, \eqref{cw2} and \eqref{cw3} need at least the weak solution of problem \eqref{physic} satisfying $u\in C^1\left(0, T; \tilde{\mathcal{H}}^{1,\fr{n}{2}}_{2,0}(\BB)\right)$, and \eqref{cw1} requires the weak solution $u\in C^1\left(0, T; \mathcal{L}^{n}_1(\BB)\right)$.

The first purpose of this paper is to improve the regularity of the weak solution and then reestablish \eqref{cw2}, \eqref{cw3} and \eqref{cw1}, our result can be stated as follow.

\begin{theorem}\label{yiyou1}
Assume $u=u(x, t)$ is the unique weak solution of problem \eqref{physic}, then
\begin{equation*}
u\in C^1\left( 0, T; \tilde{\mathcal{H}}^{1,\fr{n}{2}}_{2,0}(\BB)\right).
\end{equation*}
Moreover, \eqref{cw2}, \eqref{cw3} and \eqref{cw1} hold for any $t\in(0, T)$.
\end{theorem}

In order to introduce the second goal of the present paper naturally, we first give some notations and definitions used in \cite{dihuafei}. Throughout the paper, we define the Nehari manifold by
\begin{equation*}\begin{split}
N=\left\{u\in\tilde{\mathcal{H}}^{1,\fr{n}{2}}_{2,0}(\BB): I(u)=0, \iE |\nE u|^2\da\not=0\right\},
\end{split}\end{equation*}
which separates the two regions
\begin{equation*}\begin{split}
& N_+=\left\{u\in\tilde{\mathcal{H}}^{1,\fr{n}{2}}_{2,0}(\BB): I(u)>0\right\}\cup\{0\},\\
& N_-=\left\{u\in\tilde{\mathcal{H}}^{1,\fr{n}{2}}_{2,0}(\BB): I(u)<0\right\}.
\end{split}\end{equation*}
Moreover, it is well-known \cite{3a} that the depth of potential well $d$ may also be defined by $d=\inf_{u\in N}J(u)$. For $k\in\R$, we define the sublevels of $J$ by
$$J^{k}:=\left\{u\in\tilde{\mathcal{H}}^{1,\fr{n}{2}}_{2,0}(\BB)\ \left|\right.J(u)\leq k\right\},$$
and for all $\alpha>d$, we let
\begin{equation*}
\begin{split}
N^\alpha:=N\cap J^\alpha
         =\left\{u\in N\left|\ \|\na_{\BB}u\|_{\mathcal{L}^{\fr{n}{2}}_2(\BB)}\leq\sqrt{\frac{2\alpha(p+1)}{p-1}}\right.\right\}.
\end{split}\end{equation*}
For such $\alpha$, we further define
\begin{equation}\label{laa}
\begin{split}
\Lambda_\alpha=\sup\left\{\left.\|u\|^2_{\tilde{\mathcal{H}}^{1,\fr{n}{2}}_{2,0}(\BB)} \right|u\in N^\alpha\right\},
\end{split}\end{equation}
{\color{blue}{ the conclusion in \cite[Lemma 3.5]{dihuafei} tells us that $\Lambda_\alpha\in(0, +\infty)$ is a constant. }}

After exploring the properties of the potential wells and the invariant sets in cone Sobolev spaces, the authors in \cite{dihuafei} considered global existence, exponential decay and blow-up of weak solution to problem \eqref{physic}. Moreover, they estimated the blow-up time in some cases. The following  result can be found in \cite[Corollary 6.2]{dihuafei}.

\begin{proposition}\label{888}
Let $u(x,t)$ be the weak solution of problem \eqref{physic} with $I(u_0)<0, S(u_0)\leq 0$ and one of the following two conditions
holds:
\begin{itemize}
  \item [(i).] $J(u_0)\leq d$;
  \item [(ii).] $J(u_0)>d$ and
  \begin{equation}\label{kqd}
  \|u_0\|^2_{\tilde{\mathcal{H}}^{1,\fr{n}{2}}_{2,0}(\BB)}>\Lambda_{J(u_0)},
  \end{equation}
\end{itemize}
then $u(x,t)$ blows up in finite time, where $S(u_0)$ and $\Lambda_{J(u_0)}$ given by \eqref{shd} and \eqref{laa} respectively.
\end{proposition}

For above proposition, one may ask that whether conditions \eqref{kqd} can be removed? In other words, can we have
\begin{equation}\label{ls}
u_0\in N_-,\ S(u_0)\leq 0\Rightarrow u(x, t)\ \mbox{blows up in finite time}?
\end{equation}
Analogous question has been came up in \cite{Gazzola2005Finite}, the corresponding initial-boundary value problem to semilinear parabolic equation
\begin{equation*}
u_t-\Delta u=|u|^{p-2}u
\end{equation*}
was studied therein, here $\Delta$ is the classical Laplace operator. Dickstein et al. in \cite{Dickstein} subsequently proved that the answer is negative and there exist solutions converging to
any given steady state, with initial Nehari energy $I(u_0)$ either negative or positive. Contrary to the semilinear parabolic equation, Zhu et al. in \cite{zsc} got a positive answer to the corresponding initial-boundary value problem with pseudo-parabolic equation
\begin{equation*}
u_t-\Delta u_t-\Delta u+u=|u|^{p-2}u.
\end{equation*}
They found a sharp result about the global existence and blow-up in finite time, and proved $u_0\in N_-$ is a sufficient and necessary condition for finite time blow-up of solutions.

The second goal of present paper is to prove that \eqref{ls} holds for the weak solution to problem \eqref{physic}. We apply and extend the method used in \cite{zsc} to the case of nonlocal semilinear pseudo-parabolic equation with conical degeneration, then we get rid of the conditions \eqref{kqd} and improve the blow-up result of \cite{dihuafei}. By \eqref{cw3} we can see the sign of initial integral $S(u_0)$ may be essential for the proofs of global existence and finite time blow-up, which is caused by the nonlocal term $\frac{1}{|\mathbb{B}|}\iE |u|^{p-1}u\frac{dx_1}{x_1}dx'$. So the use of the idea is by far nontrivial because the existence of nonlocal term and because more analyses are necessary to overcome some technical points.

In fact, as claimed in \cite[Remark 5.3]{zsc}, the weak solution with initial data in the set $N_+$ may also blow up. Thus, we can not describe the blow-up phenomenon completely with the means of $N_-$ only. Let
\begin{equation*}\begin{split}
&S^-:=\left\{u_0\in\tilde{\mathcal{H}}^{1,\fr{n}{2}}_{2,0}(\BB): \mbox{there is a}\ t_0\in [0, T)\ \mbox{such that}\ u(x, t_0)\in N_- \right\},
\end{split}\end{equation*}
then one can prove that $S^-$ is strictly larger than $N_-$, see \cite[Remark 5.4]{zsc}. For some $t_0\geq0$, if the corresponding weak solution $u(x, t_0)\in N_-$, we have the following theorem of blow-up phenomenon.

\begin{theorem}\label{djtk}
Let $u(t)=u(x,t)$ be the weak solution of problem \eqref{physic} with $S(u_0)\leq 0$, then $u(t)$ blows up at finite time $T$ if and only if there exists a $t_0\in [0, T)$ such that $u(t_0)\in N_-$. Moreover, for all $t\in[t_0, T)$, $u(t)$ grows as the following sense
\begin{equation*}
\|u(t)\|^2_{\tilde{\mathcal{H}}^{1,\fr{n}{2}}_{2,0}(\BB)}\geq -2I(u(t_0))e^{(p-1)(t-t_0)}(t-t_0)+\|u(t_0)\|^2_{\tilde{\mathcal{H}}^{1,\fr{n}{2}}_{2,0}(\BB)}.
\end{equation*}
\end{theorem}  

\begin{remark}
As a special case, take $t_0=0$ in above conclusion, then for the initial data with $S(u_0)\leq 0$ we have
\begin{equation*}
I(u_0)<0,\ \mbox{i.e.}\ u_0\in N_-\ \ \Leftrightarrow\ \ \mbox{the weak solution blows up in finite time},
\end{equation*}
which gives a positive answer to \eqref{ls}.
\end{remark} 

\begin{remark}
It is worthwhile pointing out that our regularity result also can be applied to the weak solution of some related parabolic equations, and Theorem \ref{djtk} has some implications about blow-up solution for pseudo-parabolic with other types of nonlocal sources.
\end{remark}

For readers' convenience, in Section 2, we first introduce some definitions and properties of cone Sobolev spaces. The proofs of Theorem \ref{yiyou1} and \ref{djtk} will be given in Section 3 and Section 4, respectively.

\section{Preliminaries}

The detail research of manifold with conical singularities and the corresponding cone Sobolev spaces can be found in \cite{Chen2011Existence,Chen2012Cone,Coriasco2007Realizations,Schrohe2001Ellipticity,Schulze1999Boundary}. In this subsection, we shall introduce some definitions and properties of cone Sobolev spaces briefly, which is enough to make our paper readable.

Let $X$ be a closed, compact, $C^\infty$ manifold, we set $X^{\triangle}=(\bar\R_+\times X)/(\{0\}\times X)$ as a local model interpreted as a cone with the base $X$. We denote $X^\triangledown=\R_+\times X$ as the corresponding open stretched cone with $X$. An $n$-dimensional manifold $B$ with conical singularities is a topological space with a finite subset $B_0=\{b_1,\cdots,b_M\}\subset B$ of conical singularities. For simplicity, we assume that the manifold $B$ has only one conical point on the boundary. Thus, near the conical point, we have a stretched manifold $\BB$, associated with $B$.

\begin{definition}
Let $\BB=[0,1)\times X$ be the stretched manifold of the manifold $B$ with conical singularity, then for any cut-off function $\omega$, supported by a collar neighborhood of $(0,1)\times\p\BB$, the cone Sobolev space $\mathcal{H}^{m,\gamma}_{p}(\BB)$, for $m\in\mathbb{N}, \gamma\in\R$ and $1<p<+\infty$, is defined as $\mathcal{H}^{m,\gamma}_{p}(\BB)=\left\{u\in W_{loc}^{m,p}(\BB_0)\right|\left.\omega u\in\mathcal{H}^{m,\gamma}_{p}(X^\triangledown)\right\}$. Moreover, the subspace $\mathcal{H}^{m,\gamma}_{p,0}(\BB)$ of $\mathcal{H}^{m,\gamma}_{p}(\BB)$ is defined by
\begin{equation*}
  \mathcal{H}^{m,\gamma}_{p,0}(\BB)=\omega\mathcal{H}^{m,\gamma}_{p,0}(X^\triangledown)+(1-\omega)W_0^{m,p}(\BB_0),
\end{equation*}
where $W_0^{m,p}(\BB_0)$ denotes the closure of $C_0^\infty(\BB_0)$ in Sobolev spaces $W^{m,p}(\dt X)$, here $\dt X$ is a closed compact $C^\infty$ manifold of dimension $n$ that containing $\BB$ as a sub-manifold with boundary.
\end{definition}

\begin{definition}
We say $u(x)\in \mathcal{L}^\gamma_p(\BB)$ with $1<p<+\infty$ and $\gamma\in\R$ if
\begin{equation*}
  \|u\|^p_{\mathcal{L}^{\gamma}_p(\BB)}=\int_\BB x_1^n|x_1^{-\gamma}u(x)|^p\da<+\infty.
\end{equation*}
\end{definition}

Observe that if $u(x)\in\mathcal{L}^{\frac{n}{p}}_p(\BB), v(x)\in\mathcal{L}^{\frac{n}{q}}_q(\BB)$ with $p,q\in(1, +\infty)$ and $\frac{1}{p}+\frac{1}{q}=1$, then we have the following H\"{o}lder's inequality
\begin{equation*}
  \int_\BB|u(x)v(x)|\da\leq\|u\|_{\mathcal{L}^{\frac{n}{p}}_p(\BB)}\|v\|_{\mathcal{L}^{\frac{n}{q}}_q(\BB)}.
\end{equation*}

Integration by parts in cone Sobolev spaces is consistent with the one in classical Sobolev spaces, see \cite[Lemma 2.1]{dihuafei}.

\begin{lemma}
Assume that functions $u, v\in\tilde{\mathcal{H}}^{1,\fr{n}{2}}_{2,0}(\BB)$, then
\begin{equation*}
\iE v\Delta_\mathbb{B}u\da=-\iE\nabla_\mathbb{B}v\cdot\nabla_\mathbb{B}u\da.
\end{equation*}
\end{lemma}

By \cite[Propositions 2.2]{Chen2012Global}, or see \cite[Lemma 2.4]{dihuafei} directly, the corresponding cone Sobolev embedding can be stated as follows.

\begin{lemma}
For $1<p+1<\fr{2n}{n-2}$, the embedding $\tilde{\mathcal{H}}^{1,\fr{n}{2}}_{2,0}(\BB)\hookrightarrow\mathcal{H}^{0,\fr{n}{p+1}}_{p+1}(\BB)=\mathcal{L}^{\fr{n}{p+1}}_{p+1}(\BB)$ is continuous.
\end{lemma}

\section{Proof of Theorem \ref{yiyou1}}

We begin this section with following lemma.

\begin{lemma}\label{l1}
Let $u=u(x,t)$ be the weak solution of problem \eqref{physic}, then
\begin{equation*}
u\in C^1\left(0, T; \mathcal{L}^{n}_1(\BB)\right),
\end{equation*}
and \eqref{cw1} holds for all $t\in(0, T)$.
\end{lemma}

\begin{proof}
We prove our conclusion with choosing suitable test function. Let $t_0\in(0, T), t_1\in(t_0, T)$ and $\delta\in(0, \delta_0)$ with $\delta_0=\min\{t_0,\ T-t_1\}$. Setting
\begin{eqnarray*}
\chi_\delta=
\left\{
\begin{array}{ll}
0,\quad & t\in(-\infty, t_0-\delta)\cup (t_1+\delta, +\infty),\\
\frac{t-t_0+\delta}{\delta},\quad & t\in[t_0-\delta, t_0],\\
1,\quad & t\in(t_0, t_1),\\
\frac{t_1-t+\delta}{\delta},\quad & t\in[t_1, t_1+\delta],
\end{array}
\right.
\end{eqnarray*}
then by a simple calculation we can see
\begin{eqnarray*}
\chi'_\delta=
\left\{
\begin{array}{ll}
0,\ \ & t\in(-\infty, t_0-\delta)\cup (t_1+\delta, +\infty),\\
\frac{1}{\delta},\ \ & t\in[t_0-\delta, t_0],\\
0,\ \ & t\in(t_0, t_1),\\
-\frac{1}{\delta},\ \ & t\in[t_1, t_1+\delta].
\end{array}
\right.
\end{eqnarray*}
Taking $\chi_\delta$ as a test function in \eqref{rj} and integrating from $0$ to $T$ we see
\begin{equation*}\begin{split}
&\int_0^T\left(\iE u_t\chi_\delta\da+\iE\nE u_t\nE\chi_\delta\da+\iE\nE u\nE\chi_\delta\da\right)dt\\
&=\int_0^T\left(\iE|u|^{p-1}u\chi_\delta\da-\frac{1}{\BB}\iE|u|^{p-1}u\da\iE\chi_\delta\da\right)dt.
\end{split}\end{equation*}
Combining the fact $\chi_\delta(0)=\chi_\delta(T)=\nE\chi'_\delta(t)=\nE\chi_\delta(t)=0$, and using integration by parts with respect to $t$ we arrive at
\begin{equation*}\begin{split}
\int_0^T\iE u\chi'_\delta \da dt=0,
\end{split}\end{equation*} 
i.e.,
\begin{equation*}\label{hah23}
-\frac{1}{\delta}\int_{t_0-\delta}^{t_0}\iE u\da dt+\frac{1}{\delta}\int_{t_1+\delta}^{t_1}\iE u\da dt=0.
\end{equation*}
It follows from \eqref{zjsk} that the left side of above equality converges $\iE u(t_1)\da-\iE u(t_0)\da$ as $\delta\rightarrow 0$, then
\begin{equation*}\begin{split}
\iE \left(u(t_1)-u(t_0)\right)\da=0.
\end{split}\end{equation*}
Upon division by $t_1-t_0$ and taking limits $t_1\rightarrow t_0$, we infer from the arbitrariness of $t_0\in(0, T)$ that $\iE u(t)\da\in C^1(0, T)$ with $\frac{d}{dt}\iE u(t)\frac{dx_1}{x_1}dx'=0$.

\end{proof}

Now, we prove our main theorem.

\begin{proof}[Proof of Theorem \ref{yiyou1}]
Let $u=u(x, t)$ be the weak solution of problem \eqref{physic}. We first claim that
\begin{equation}\label{l2}
u\in C^1\left(0, T; \mathcal{L}^{\fr{n}{2}}_2(\BB)\right).
\end{equation}
Indeed, for all $t, s\in(0, T), s\not=t$, we can see
\begin{equation*}\begin{split}
\left|\frac{\|u(t)\|_{\mathcal{L}^{\fr{n}{2}}_2(\BB)}^2-\|u(s)\|_{\mathcal{L}^{\fr{n}{2}}_2(\BB)}^2}{t-s}\right|&=\left|\frac{1}{t-s}\iE(u(t)-u(s))(u(t)+u(s))\da\right|\\
&\leq\sup_{t\in(0, T)}\sup_{x\in\BB}|u(t)+u(s)|\iE\frac{u(t)-u(s)}{t-s}\da.
\end{split}\end{equation*}
Let $s\rightarrow t$ in above inequality, then \eqref{zjsk2} and Lemma \ref{l1} leads to \eqref{l2} with
\begin{equation}\label{scs}
\frac{d}{dt}\|u(t)\|_{\mathcal{L}^{\fr{n}{2}}_2(\BB)}^2\leq0.
\end{equation}

Then we can use similar way to get
\begin{equation}\label{kns}
u\in C^1\left(0, T; \tilde{\mathcal{H}}^{1,\fr{n}{2}}_{2,0}(\BB)\right).
\end{equation}
In fact, for all $t, s\in(0, T), s\not=t$ again, using integration by parts and the H\"{o}lder inequality in cone Sobolev spaces we can obtain
\begin{equation*}\begin{split}
&\left|\frac{\|\na_{\BB}u(t)\|_{\mathcal{L}^{\fr{n}{2}}_2(\BB)}^2-\|\na_{\BB}u(s)\|_{\mathcal{L}^{\fr{n}{2}}_2(\BB)}^2}{t-s}+2\iE u_t(t)\Delta_\mathbb{B}u(t)\da\right|\\
&=\left|\frac{1}{t-s}\iE\nE(u(t)-u(s))\nE(u(t)+u(s))\da+2\iE u_t(t)\Delta_\mathbb{B}u(t)\da\right|\\
&=\left|\iE\frac{u(t)-u(s)}{t-s}\left(-\Delta_\mathbb{B}\right)(u(t)+u(s))\da+2\iE u_t(t)\Delta_\mathbb{B}u(t)\da\right|\\
&=\left|\iE\left[\left(\frac{u(t)-u(s)}{t-s}-u_t(t)\right)\left(-\Delta_\mathbb{B}\right)(u(t)+u(s))+ u_t(t)\left(\Delta_\mathbb{B}u(t)-\Delta_\mathbb{B}u(s)\right)\right]\da\right|\\
&\leq\left\|\frac{u(t)-u(s)}{t-s}-u_t(t)\right\|_{\mathcal{L}^{\fr{n}{2}}_2(\BB)}\left\|\Delta_\mathbb{B}(u(t)+u(s))\right\|_{\mathcal{L}^{\fr{n}{2}}_2(\BB)}\\
&\quad+\left\|u_t(t)\right\|_{\mathcal{L}^{\fr{n}{2}}_2(\BB)}\left\|\Delta_\mathbb{B}u(t)-\Delta_\mathbb{B}u(s)\right\|_{\mathcal{L}^{\fr{n}{2}}_2(\BB)}.
\end{split}\end{equation*}
Since there holds \eqref{l2}, then $\left\|\frac{u(t)-u(s)}{t-s}-u_t(t)\right\|_{\mathcal{L}^{\fr{n}{2}}_2(\BB)}\rightarrow0\ \ \mbox{as}\ \ s\rightarrow t$. Hence, we can let $s\rightarrow t$ to get
\begin{equation}\label{slx}
\frac{d}{dt}\|\na_{\BB}u(t)\|_{\mathcal{L}^{\fr{n}{2}}_2(\BB)}^2=-2\iE u_t(t)\Delta_\mathbb{B}u(t)\da.
\end{equation}
This togethers with \eqref{l2} deduces to \eqref{kns}, then use the first equation in \eqref{physic} it is easy to get \eqref{cw3}. Furthermore, by the continuous embedding from $\tilde{\mathcal{H}}^{1,\fr{n}{2}}_{2,0}(\BB)\hookrightarrow \mathcal{L}^{\fr{n}{p+1}}_{p+1}(\BB)$ we see $u\in C^1\left((0, T); \tilde{\mathcal{H}}^{1,\fr{n}{2}}_{2,0}(\BB)\right)\hookrightarrow C^1\left((0, T); \mathcal{L}^{\fr{n}{p+1}}_{p+1}(\BB)\right)$.

Finally, we prove \eqref{cw2}. We aim to show
\begin{equation}\label{slx2}
\frac{d}{dt}\|u(t)\|_{\mathcal{L}^{\fr{n}{p+1}}_{p+1}(\BB)}^{p+1}=(p+1)\iE u_t\left(|u|^{p-1}u-\frac{1}{|\mathbb{B}|}\iE |u|^{p-1}u\frac{dx_1}{x_1}dx'\right)\da,
\end{equation}
because if it holds, then \eqref{slx} and the first equation in \eqref{physic} give \eqref{cw2} immediately.

By \eqref{cw1} we know
\begin{equation*}\begin{split}
(p+1)\iE u_t\left(|u|^{p-1}u-\frac{1}{|\mathbb{B}|}\iE |u|^{p-1}u\frac{dx_1}{x_1}dx'\right)\da&=(p+1)\iE u_t|u|^{p-1}u\frac{dx_1}{x_1}dx'\\
&=\frac{d}{dt}\|u(t)\|_{\mathcal{L}^{\fr{n}{p+1}}_{p+1}(\BB)}^{p+1},
\end{split}\end{equation*}
so we complete our proof.
\end{proof}

\section{Proof of Theorem \ref{djtk}}

We first claim that $N_-$ is an invariant set if $S(u_0)\leq0$.

\begin{lemma}\label{bubianj}
Let $u(t)$ be the weak solution of problem \eqref{physic} with $u_0\in S^-, S(u_0)\leq0$, then for all $t\in[t_0, T),\ u(t)\in N_-$.
\end{lemma}

\begin{proof}
By the definitions of $J(u(t)), I(u(t))$, \eqref{cw2}, \eqref{cw3}, \eqref{scs} and \eqref{kns} we have that for all $t\in [t_0, T)$,
\begin{equation*}\begin{split}
\frac{d}{dt}I(u(t))&=\frac{d}{dt}\left((p+1)J(u(t))-\frac{p-1}{2}\|\na_{\BB}u(t)\|_{\mathcal{L}^{\fr{n}{2}}_2(\BB)}^2\right)\\
&=\frac{d}{dt}\left((p+1)J(u(t))-\frac{p-1}{2}\|u(t)\|_{\tilde{\mathcal{H}}^{1,\fr{n}{2}}_{2,0}(\BB)}^2+\frac{p-1}{2}\|u(t)\|_{\mathcal{L}^{\fr{n}{2}}_2(\BB)}^2\right)\\
&\leq-(p+1)\|u(t)\|_{\tilde{\mathcal{H}}^{1,\fr{n}{2}}_{2,0}(\BB)}^2+(p-1)I(u(t))+\frac{(p-1)S(u_0)}{|\BB|}\|u(t)\|^p_{\mathcal{L}^{\fr{n}{p}}_{p}(\BB)}\\
&\leq (p-1)I(u(t)).
\end{split}\end{equation*}
Therefore, it follows from the Gronwall inequality that
\begin{equation}\label{5.2}
I(u(t))\leq I(u(t_0))e^{(p-1)(t-t_0)},\ \ \forall t\in(t_0, T).
\end{equation}
Since $I(u(t_0))<0$, then our claim is true.
\end{proof}

\begin{proof}[Proof of Theorem \ref{djtk}]
Let $u=u(x, t)$ be the weak solution of problem \eqref{physic} with $S(u_0)\leq0$, $T\in(0, +\infty]$ be the maximal existence time of $u$.

We first the sufficiency, i.e., we prove that $u$ blows up at finite time $T$ under $u_0\in S^-$. By Lemma \ref{bubianj} we know $I(u(t))<0$ for all $t\in[t_0, T)$, so we can infer from Proposition \ref{888}(i) that if $J(u(t_0))<d$, then take $t_0$ as initial time, $u$ blows up at finite time $T$. Therefore, to finish the proof of this theorem, we need only consider the case that
\begin{equation*}
d\leq J(u(t))\leq J(u(t_0)),\quad\forall t\in [t_0, T).
\end{equation*}
Arguing with contradiction, suppose that $u$ exists globally, i.e., $T=+\infty$. Above inequalities and \eqref{cw3} suggest that $J(u(t))$ is non-increasing and bounded on $[t_0, +\infty)$, so the limit $A:=\lim_{t\rightarrow\infty}J(u(t))$ exists and there holds
\begin{equation}\label{xuyao}
\int_{t_0}^\infty\|u_t(t)\|^2_{\tilde{\mathcal{H}}^{1,\fr{n}{2}}_{2,0}(\BB)}=J(u(t_0))-A,
\end{equation}
where $A=\frac{1}{2}A_1-\frac{1}{p+1}A_2$ and
\begin{equation}\label{a2}
A_1:=\lim_{t\rightarrow\infty}\|\na_{\BB}u(t)\|_{\mathcal{L}^{\fr{n}{2}}_2(\BB)}^2<+\infty,\ \ A_2:=\lim_{t\rightarrow\infty}\|u(t)\|^{p+1}_{\mathcal{L}^{\fr{n}{p+1}}_{p+1}(\BB)}<+\infty.
\end{equation}
Moreover, \eqref{xuyao} yields that $\int_{t_0}^\infty\|u_t(t)\|^2_{\tilde{\mathcal{H}}^{1,\fr{n}{2}}_{2,0}(\BB)}<+\infty$, which further deduces that there exists a diverging sequence $\{t_n\}$ such that
\begin{equation*}
\lim_{n\rightarrow\infty}\|u_t(t_n)\|^2_{\tilde{\mathcal{H}}^{1,\fr{n}{2}}_{2,0}(\BB)}=0.
\end{equation*}
By \eqref{qjxd} we know $u\in L^\infty ([t_0, \infty), \tilde{\mathcal{H}}^{1,\fr{n}{2}}_{2,0}(\BB))$. Thus, by \eqref{cw2} we can obtain that
\begin{equation*}\begin{split}
\left|I(u(t_n))\right|&=\left|(u_t(t_n), u(t_n))+(\nE u_t(t_n), \nE u(t_n))+\frac{S(u_0)}{|\BB|}\|u(t_n)\|^p_{\mathcal{L}^{\fr{n}{p}}_{p}(\BB)}\right|\\
&\leq\|u_t(t_n)\|_{\tilde{\mathcal{H}}^{1,\fr{n}{2}}_{2,0}(\BB)}\ \|u(t_n)\|_{\tilde{\mathcal{H}}^{1,\fr{n}{2}}_{2,0}(\BB)}+
\frac{|S(u_0)|}{|\BB|}\|u(t_n)\|^p_{\mathcal{L}^{\fr{n}{p}}_{p}(\BB)}\\
&\leq\|u_t(t_n)\|_{\tilde{\mathcal{H}}^{1,\fr{n}{2}}_{2,0}(\BB)}\ \|u(t_n)\|_{\tilde{\mathcal{H}}^{1,\fr{n}{2}}_{2,0}(\BB)}+
C(|\BB|)|S(u_0)|\ \|u(t_n)\|^{p+1}_{\mathcal{L}^{\fr{n}{p+1}}_{p+1}(\BB)},
\end{split}\end{equation*}
where we also used the H\"{o}lder inequality in cone Sobolev spaces and $C(|\BB|)$ is a positive constant with respect to $|\BB|$. Taking $n\rightarrow\infty$, we can see
\begin{equation*}
\left|I(u(t_n))\right|\leq C(|\BB|)|S(u_0)|A_2<+\infty,
\end{equation*}
here $A_2$ is the constant given by \eqref{a2}. However, by \eqref{5.2} we know as $n\rightarrow\infty$ there holds that
\begin{equation*}
|I(u(t_n))|\geq-I(u(t_0))e^{(p-1)t_n}\rightarrow+\infty.
\end{equation*}
So a contradiction occurs. Therefore, $T<+\infty$ and $u$ blows up in finite time.

The exponential growth of $u(t)$ for all $t\in[t_0, T)$ comes from \eqref{cw3} and \eqref{5.2}, that is,
\begin{equation*}\begin{split}
\frac{d}{dt}\|u(t)\|^2_{\tilde{\mathcal{H}}^{1,\fr{n}{2}}_{2,0}(\BB)}&=-2I(u(t))-\frac{2S(u_0)}{|\BB|}\|u(t)\|^p_{\mathcal{L}^{\fr{n}{p}}_{p}(\BB)}\\
&\geq -2I(u(t_0))e^{(p-1)(t-t_0)},
\end{split}\end{equation*}
integrating from $t_0$ to $t$, we arrive at the desired result.

Finally, we aim to obtain the necessity, i.e.,
\begin{equation*}
u(x, t)\ \mbox{blows up at finite time}\ T \Rightarrow u_0\in S^-.
\end{equation*}
In this end, we will prove the corresponding equivalent proposition that
\begin{equation}\label{djmt}
u_0\not\in S^-\ \Rightarrow\ T=+\infty.
\end{equation}
By $u_0\not\in S^-$ we know $I(u(t))\geq0$ for all $t\in[0, T)$. If $I(u(t))>0$ for all $t\in[0, T)$, then we can infer from \cite[Theorem 4.1 and 5.1]{dihuafei} that $T=+\infty$ under $J(u_0)\leq d$. While when $I(u(t))=0$ for all $t\in[0, T)$, then $I(u_0)=0$, this combines $u_0\in\tilde{\mathcal{H}}^{1,\fr{n}{2}}_{2,0}(\BB)\backslash\{0\}$ imply that $u_0\in N$, then $J(u_0)\geq d$. If $J(u_0)=d$, then \cite[Theorem 5.1]{dihuafei} tells us the weak solution exists globally. Hence, in order to claim \eqref{djmt} we need only prove that
\begin{equation*}
J(u_0)>d,\ I(u(t))\geq0\ \mbox{for all}\ t\in[0, T)\Rightarrow T=+\infty.
\end{equation*}
By the fact that $J(u(t))\leq J(u_0)$, it holds that
\begin{equation*}\begin{split}
\frac{p-1}{2(p+1)}\|u(t)\|^2_{\tilde{\mathcal{H}}^{1,\fr{n}{2}}_{2,0}(\BB)}&\leq\frac{p-1}{2(p+1)}\|u(t)\|^2_{\tilde{\mathcal{H}}^{1,\fr{n}{2}}_{2,0}(\BB)}
+\frac{1}{p+1}I(u(t))\\
&=J(u(t))\leq J(u_0).
\end{split}\end{equation*}
This implies that $\|u(t)\|^2_{\tilde{\mathcal{H}}^{1,\fr{n}{2}}_{2,0}(\BB)}$ is uniformly bounded on $[0, T)$, so $u$ exists globally.
\end{proof}

\section*{Acknowledgement}

The authors convey many thanks to the anonymous reviewers for their helpful suggestions and positive comments, which improve this paper and encourages the authors greatly.



\end{document}